\documentclass{amsart}

\usepackage[utf8]{inputenc}
\usepackage[dvipsnames]{xcolor}
\usepackage{tikz}
\usepackage{tikz-cd}
\usepackage{amsmath}
\usepackage{mathrsfs} 
\usepackage{scalerel}
\usepackage{amssymb}
\usepackage{bm}
\usepackage{graphicx}
\usepackage[centertags]{amsmath}
\usepackage{amsfonts}
\usepackage{amsthm}
\usepackage{enumerate}
\usepackage{xcolor}
\usepackage{esint}
\usepackage{pgfplots}
\usepackage{tikz}
\usepackage{tikz-cd}
\usepackage{adjustbox}

 \newtheorem{thm}{Theorem}[section]
 \newtheorem{cor}[thm]{Corollary}
 \newtheorem{lem}[thm]{Lemma}
 \newtheorem{prop}[thm]{Proposition}
 \theoremstyle{definition}
 \newtheorem{defn}[thm]{Definition}
 \theoremstyle{remark}
 \newtheorem{rem}[thm]{Remark}
 
 \numberwithin{equation}{section}

\newcommand{\spanS}{{\rm span}}
\newcommand{\signS}{{\rm sign}}
\newcommand{\suppS}{{\rm supp}}

\begin{document}

%-------------------------------------------------------------------------
% editorial commands: to be inserted by the editorial office
%
%\firstpage{1} \volume{228} \Copyrightyear{2004} \DOI{003-0001}
%
%
%\seriesextra{Just an add-on}
%\seriesextraline{This is the Concrete Title of this Book\br H.E. R and S.T.C. W, Eds.}
%
% for journals:
%
%\firstpage{1}
%\issuenumber{1}
%\Volumeandyear{1 (2004)}
%\Copyrightyear{2004}
%\DOI{003-xxxx-y}
%\Signet
%\commby{inhouse}
%\submitted{March 14, 2003}
%\received{March 16, 2000}
%\revised{June 1, 2000}
%\accepted{July 22, 2000}
%
%
%
%---------------------------------------------------------------------------
%Insert here the title, affiliations and abstract:
%

\title[Factorization property in rearrangement invariant spaces]
 {Factorization property in rearrangement invariant spaces}

%----------Author 1

\author[Kh. V. Navoyan]{Kh. V. Navoyan}

\address{
Department of Mathematics and Statistics\\
Thompson Rivers University\\
Kamloops\\
BC V2C 0C8\\
Canada}

\email{khazhakanush@gmail.com}

%\thanks{This work was completed with the support of our
%\TeX-pert.}
%----------Author 2
%\author{A Second Author}
%\address{The address of\br
%the second author\br
%sitting somewhere\br
%in the world}
%\email{dont@know.who.knows}
%----------classification, keywords, date

\subjclass{Primary 46B09, 46B25, 46B28, 47A68}

\keywords{Factorization of operators, classical Banach spaces, unconditional basis, rearrangement invariant space, Haar basis, faithful Haar system, Rademacher functions, block basis, weakly null sequence, operator with a large diagonal}

\date{March 10, 2023}
%----------additions
%\dedicatory{To my boss}
%%% ----------------------------------------------------------------------

\begin{abstract}
Let $X$ be a Banach space with a basis $(e_k)_k$ and biorthogonals $(e^\ast_k)_k$. An operator on $X$ is said to have a {\it large diagonal} if $\inf\limits_{k} |e_k^\ast(T(e_k))| > 0$. The basis $(e_k)_k$ is said to have the {\it factorization property} if the identity factors through any operator with a large diagonal. Under the assumption that the Rademacher sequence is weakly null, we study the factorization property of the Haar system in a Haar system space. A Haar system space is the completion of the span of characteristic functions of dyadic intervals with respect to a rearrangement invariant norm. We show that every bounded operator with a large diagonal on a Haar system space is approximatively a factor of some diagonal operator with a large diagonal. Moreover, when the Haar system is an unconditional basis for a Haar system space, it has the factorization property.   
\end{abstract}

%%% ----------------------------------------------------------------------
\maketitle
%%% ----------------------------------------------------------------------
%\tableofcontents
\allowdisplaybreaks

\section{Introduction}

\noindent In his paper \cite{Andrew} on perturbations of Schauder bases, Andrew shows the following: If $T$ is an operator on $L^p$, $1 <p < \infty$, which has the property that as a matrix with respect to the Haar basis the diagonal is bounded away from zero ($T$ has a {\it large diagonal}), then the image $TL^p$ has a complemented subspace isomorphic to $L^p$. Actually he showed that the identity on $L^p$ factors through $T$. In our paper we generalize this fact to rearrangement invariant (r.i.) Haar system spaces. Let us recall some definitions.

\begin{defn}\cite[Definition 1.1]{Andrew} Let $C > 0$. A normalized Schauder basis $\{x_n\}$ for a Banach space $X$ is {\it$C$-perturbable} if for each $\delta > 0$ and each bounded linear operator $T: X\to X$ satisfying $\|Tx_n - x_n\| < C - \delta$ for all $n$, the space $T(X)$ contains an isomorphic copy of $X$. The largest $C$ for which $\{x_n\}$ is $C$-perturbable is called the {\it perturbation constant} of the basis $\{x_n\}$.  
\end{defn}

\begin{defn}\cite{LLM}
Let $X$ be a Banach space with a normalized Schauder basis $(e_k)_k$, and let $(e_k^\ast)_k$ be the sequence of the biorthogonals.
\begin{itemize}

\item[(i)] If an operator $T$ on $X$ satisfies $\inf\limits_{k} |e_k^\ast(T(e_k))| > 0$ then we say that $T$ has a {\it large diagonal}.

\item[(ii)] An operator $T$ on $X$ satisfying $e_m^\ast (T(e_k)) = 0$ whenever $k\neq m$, is called a {\it diagonal operator}.

\item[(iii)] We say that {\it the identity operator $I_X$ on $X$ factors through $T$} if there are bounded linear operators $R, S: X\to X$ with $I_X = STR$.

\item[(iv)] We say that the basis $(e_k)_k$ has the {\it factorization property} if whenever $T: X\to X$ is an operator with a large diagonal, then the identity of $X$ factors through $T$.
\end{itemize}
\end{defn}

\noindent The problem of factoring the identity through an operator with a large diagonal originates from a study by Pe\l czy\'nski \cite{Pel.}, where he shows that every infinite dimensional subspace of $\ell^p$, $1 \le p < \infty$, and $c_0$ contains a further subspace which is complemented and isomorphic to the whole space. Actually he deduces that $\ell^p$ (or $c_0$) is {\it prime}, meaning that every infinite dimensional complemented subspace of $\ell^p$ (correspondingly, of $c_0$) is isomorphic to $\ell^p$ ($c_0$).  

\noindent As noted in \cite[Remark 2.4]{StR}, it is easy to see that a basis that has the factorization property is also $1$-perturbable, but as shown in \cite[Propositions 2.5 and 2.7]{StR}, the James space with its shrinking basis is perturbable, but does not have the factorization property. 

\noindent Also, in \cite{LLM}, it was shown that for $1 \le p, q < \infty$, the identity on the biparametric Hardy space $H^p(H^q)$ factors through every operator $T: H^p(H^q) \to H^p(H^q)$ with a large diagonal. In addition, in \cite{LLM} the authors show that there exists a Banach space with an unconditional basis which does not have the factorization property.     

\noindent In \cite{StR} Lechner et al. introduced the concept of {\it strategical reproducibility} of bases. This led to simplified proofs of existing results and to the new result that the Haar basis on $L^1[0, 1]$ has the factorization property.
They also proved that the biparametric Haar basis on $L^1([0, 1]^2)$ and the tensor product of the $\ell^p$ unit vector basis with the Haar basis have the factorization property.

\noindent There is another concept, called primarity of a Banach space which is closely related to the property of factorization. A Banach space $X$ is called {\it primary}, if for every bounded projection $P: X\to X$, either $P(X)$ or $(I - P)(X)$ is isomorphic to $X$. When a Banach space is primary, it follows that either $P$ has a large diagonal or $I - P$ has a large diagonal on a ``large" subsequence of the basis $(e_i)_{i=1}^\infty$ of the Banach space $X$. As observed in \cite[Proposition 2.4]{Prim}, if a Banach space $X$ has the {\it accordion} property, {\it i.e.,} is isomorphic to its $\ell_p$ sum, $1 < p < \infty$, and has a basis with the factorization property, then $X$ is primary. In particular, Enflo proved that $L^p$, $1 \le p < \infty$, is primary, by showing that for every operator $T: L^p \to L^p$ the identity operator factors either through $T$ or $I - T$ \cite{AEO, ME}. Further, Capon in \cite{Cap} proved factorization and primarity theorems for the mixed norm spaces $L^p(L^q)$, $1 < p, q < \infty$. Correspondingly, for the spaces $H^1$ and $BMO$, this type of theorems were proved by M\"uller in \cite{M.}. Further, in \cite{Prim} the authors proved that the space $L^1(L^p)$ is primary for $1 < p < \infty$.

\noindent We recall the definition of a \textit {Haar system space} and preliminary definitions related to it.

\noindent Let the collection of all dyadic intervals in $[0, 1)$ be denoted by $\mathcal{D}$, that is,

\begin{equation}\notag
 \mathcal{D} = \Big\{\Big[{\frac{i - 1}{2^j}}, {\frac{i}{2^j}}\Big): j \in \mathbb{N}\cup \{0\}, 1\le i \le 2^j\Big\}.
\end{equation}

\noindent For $n = 0, 1, 2, \ldots$, let $\mathcal{D}_n = \{J\in \mathcal{D}: |J| = 2^{-n}\}$, and $\mathcal{D}^n = \{J\in \mathcal{D}: |J| \ge 2^{-n}\}$ or $\mathcal{D}^n = \bigcup\limits_{j=0}^n \mathcal{D}_j$, and let $\mathcal{D}^+ = \mathcal{D}\cup \{\emptyset\}$. The enumeration $\imath: \mathcal{D}^+\to \mathbb{N}$ is given by

\begin{equation}\notag
\imath(\emptyset) = 1,\hspace{.5cm}\imath\bigg(\Big[{\frac{i - 1}{2^j}}, {\frac{i}{2^j}}\Big)\bigg) = 2^j +i. 
\end{equation} 
Define $I_j = \imath^{-1}(j)$ for $j\in \mathbb{N}$. Note that this orders $\mathcal{D}^+$ into the sequence $(I_j)$ with the property that $I_1 =\emptyset$, $I_2 = [0, 1]$ and $I_j = I_{2j-1} \cup I_{2j}$, $j\ge 2$.

\begin{defn}\label{ord}
The Haar function $h_I$ with support $I = \Big[{\dfrac{i - 1}{2^j}}, {\dfrac{i}{2^j}}\Big)$ is defined by $h_I = \chi_{I^+} - \chi_{I^-}$, where the dyadic intervals $I^+$ and $I^-$ are given as follows:

\begin{equation}\notag
I^+ = \Big[{\frac{i - 1}{2^j}}, {\frac{2i-1}{2^{j+1}}}\Big), \;\; I^- = \Big[{\frac{2i - 1}{2^{j+1}}}, {\frac{i}{2^j}}\Big).
\end{equation}
\end{defn}

\noindent We also index the Haar functions by natural numbers and put $h_n = h_{I_n}$, where $h_1 = h_\emptyset = \chi_{[0, 1)}$. Using this enumeration the Haar functions form a monotone Schauder basis of $L^p[0, 1]$, $1 \le p < \infty$ \cite[Example (v), page 186]{FHH}.

\noindent Note that for $j\ge 2$ it follows that
\begin{align*}
\suppS(h_{2j-1}) = [ h_j = 1],\;\mathrm{and}\;
\suppS( h_{2j}) = [ h_j = -1].\\
\end{align*}

\noindent Let the notation $\spanS(\cdot)$ denote the linear span.

\begin{defn}\label{i.ii}\cite[Definition 2.12]{Prim}
A {\it Haar system space} $X$ is the completion of $Z = \spanS(\{h_L: L\in \mathcal{D}^+\}) = \spanS(\{\chi_I: I\in \mathcal{D}\})$ under a norm $\|\cdot\|$ that satisfies the following properties.
 
\begin{itemize}
\item[(i)] If $f$ and $g$ are in $Z$ and $|f|$ and $|g|$ have the same distribution then $\|f\| = \|g\|$.

\item[(ii)] $\|\chi_{[0, 1)}\| = 1.$
\end{itemize}
\end{defn}

\noindent The following proposition states important properties for Haar system spaces.

\begin{prop}\label{Prim.} \cite[Proposition 2.13]{Prim}
Let $X$ be a Haar system space.

\begin{itemize}

\item[(a)] For every $f \in Z = \spanS(\{\chi_I: I\in \mathcal{D}\})$ we have $\|f\|_{L_1} \le \|f\|_X \le \|f\|_{L_\infty}$.

\noindent Therefore, $X$ can be naturally identified with a space of measurable functions on $[0, 1]$ and $\overline{Z}^{\|\cdot\|_{L_{\infty}}}\subset X \subset L_1$.

\item[(b)] We identify $Z$ with a subspace of $X^\ast$ via the usual representation:

\begin{equation}\notag
\Psi: Z\to X^\ast, \;\; g\to x_g^\ast, \;\;\mbox{with} \;\; x_g^\ast (f) = \int\limits_0^1 f(x) g(x) \mathrm{d}x, \;\;\mbox{for}\;\; f\in X,
\end{equation}

\noindent and denote the dual norm by $\|\cdot\|_\ast$. Then $(i)$ and $(ii)$ of Definition \ref{i.ii} are also satisfied for $\|\cdot\|_\ast$. Thus the closure of $Z$ with respect to $\|\cdot\|_\ast$ is also a Haar system space.

\item[(c)] The Haar system, in the above introduced linear order, is a monotone Schauder basis of $X$. 

\item[(d)] For a finite union $A$ of elements of $\mathcal{D}$ we put $\mu_A = \|\chi_A\|_X^{-1}$ and $\nu_A = \|\chi_A\|^{-1}_{X^\ast}$. Then, $\mu_A \nu_A = |A|^{-1}$, where $|\cdot|$ for a measurable set $A$ denotes the Lebesgue measure of $A$. In particular, $(\nu_L h_L, \mu_L h_L)_{L \in \mathcal{D}^+}$ is a biorthogonal system in $X^\ast \times X$.

\end{itemize}
\end{prop}

\begin{defn}\cite{Prim}
A (non-zero) Banach space $X$ of measurable scalar functions on $[0, 1)$ is called rearrangement invariant if the following conditions hold true.

\begin{itemize}
\item[(i)]
Whenever $f\in X$ and $g$ is a measurable function with $|g|\le|f|$ a.e. then $g\in X$ and $\|g\|_X\le\|f\|_X$.

\item[(ii)]
If $f$ and $g$ are in $X$ and they have the same distribution then $\|f\|_X = \|g\|_X$.
\end{itemize}
\end{defn}

\noindent By \cite[Proposition 2.c.1]{LT}, every separable rearrangement invariant space $X$ is a Haar system space, and by \cite{EH} (see also \cite[Proposition 2.c.5]{LT}) the Haar basis is unconditional if and only if $X$ does not contain uniformly isomorphic copies of $\ell_1^n$ and $\ell_\infty^n$, for all $n$, on disjoint vectors with the same distribution function. 

\vspace{.1 in}

\noindent In our paper, we study the factorization property of the Haar system in a Haar system space.
In Section 2, we prove some important properties of Haar system spaces which we need later. 
In Section 3, we prove the following. We assume that the sequence of Rademacher functions in a Haar system space is weakly null. Then we show that for every operator $T$ with a large diagonal there exists a diagonal operator with a large diagonal which approximately factors through $T$. Finally, in Section 4, we assume that the Haar basis is unconditional and we also obtain the factorization of the identity on a Haar system space through the bounded linear operator $T$.

%%%%%%%%%%%
%%%%%%%%%%%
%%%%%%%%%%%

\section{Faithful Haar System}

\vspace{.2 in}

\begin{defn}\label{FHS}
\noindent A {\it {faithful Haar system}} is a sequence $(\tilde h_j)_{j=1}^\infty$ of $\{0, \pm 1\}$ valued functions with $\Tilde{h}_1 = h_1 = 1_{[0,1]}$ and $(\tilde h_j)_{j=2}^\infty$ is in $\spanS(h_I: I \in \mathcal{D})$ so that:

\begin{itemize}
\item[(i)]

$\suppS(\tilde h_2) = [0,1]$,

\item[(ii)] If $j\ge 3$ and $j = 2k - 1$, then $\suppS(\tilde h_j) = [\tilde h_k = 1]$, and if  
$j\ge 4$ and $j = 2k$, then $\suppS(\tilde h_j) = [\tilde h_k = -1]$. 
\end{itemize}

\end{defn}

\begin{rem}\label{aftFHS}
Since $(\tilde h_j)_{j=2}^\infty \subset \spanS(h_I: I \in \mathcal{D})$ it follows for every $j\ge2$ that $\int\limits_0^1 \tilde h_j(t)dt = 0$ and since $\tilde h_j$ is $\{0, \pm 1\}$ valued, it follows that
\begin{equation}\label{cond}
|[\tilde h_j = 1]| = |[\tilde h_j = -1]| =\dfrac{1}{2}|\suppS (\tilde h_j)|.
\end{equation}
From $(i)$ and $(ii)$ of Definition \ref{FHS} it follows by induction for $j\ge 2$ that the measure of the support of $\tilde h_j$ is equal to the measure of $I_j$. It also follows for any $j\in \mathbb{N}$ that any linear combination $\sum\limits_{i = 1}^j \alpha_i \tilde h_i$ is constant on the support of $\tilde h_{j+1}$.
\end{rem}

\noindent From this observation we deduce the following Proposition.

\begin{prop}\label{Inducc}
Let $(\Tilde{h}_i)_{i=1}^\infty$ be a faithful Haar system. Let $f_j$ and $g_j$ be measurable functions in the rearrangement invariant space $F$, defined as

\begin{equation}
    f_j = \sum_{i=1}^j \xi_i h_i, \hspace{1cm}
    g_j = \sum_{i=1}^j \xi_i \Tilde{h}_i.
\end{equation}

\noindent Then $f_j$ and $g_j$ have the same distribution.
\end{prop}

\noindent Proposition \ref{Inducc} implies the following corollary.

\begin{cor}\label{Cor}
Let $(\Tilde{h}_i)_{i=1}^\infty$ be a faithful Haar system. Then $(h_i)_{i=1}^\infty$ and $(\Tilde{h}_i)_{i=1}^\infty$ are isometrically equivalent in a Haar system space X.
\end{cor}

\noindent We notice that the sequence $\Big({\dfrac{\Tilde{h}_j}{\|h_j\|_F}}\Big)_{j=1}^\infty$ is a block basis for $\Big({\dfrac{h_j}{\|h_j\|_F}}\Big)_{j=1}^\infty$, so that its biorthogonals are $\Big({\dfrac{\Tilde{h}_j}{\|h_j\|_{F^\ast}}}\Big)_{j=1}^\infty$ and $\Big\langle {\dfrac{\Tilde{h}_j}{\|h_j\|_F}}, {\dfrac{\Tilde{h}_j}{\|h_j\|_{F^\ast}}}\Big\rangle = 1$ for all natural $j$.

\vspace{.1 in}

\begin{prop}\cite[Proposition 3.7]{StR}
Assume that $X$ is a Banach space with a basis $(e_n)$, whose basis constant is $\lambda \ge 1$ and biorthogonal functionals $(e_n^\ast)$. Let $(b_n)$ and $(b_n^\ast)$ be block bases of $(e_n)$ and $e_n^\ast$, respectively, so that $b^\ast_m(b_n) = \delta_{m,n}$, for $m, n \in \mathbb{N}$, and so that for some $C\ge 1$ it follows that
$$
\Big\|\sum\limits_{j=1}^\infty \xi_j b_j\Big\|_X\le \sqrt{C}\Big\|\sum\limits_{j=1}^\infty \xi_j e_j\Big\|_X,\;\;\mathrm{and}\;\;\Big\|\sum\limits_{j=1}^\infty \xi_j b^\ast_j\Big\|_{X^\ast}\le \sqrt{C}\Big\|\sum\limits_{j=1}^\infty \xi_j e^\ast_j\Big\|_{X^\ast},
$$
for all $(\xi_j)\in c_{00}$.

\noindent Then $Y=\overline{\spanS(b_j:j\in \mathbb{N})}$ is a complemented subspace of $X$ and
$$
P:X\to Y, \;\; x\mapsto \sum\limits_{n=1}^\infty b_n^\ast(x)b_n
$$
is well defined and a bounded projection onto $Y$ with $\|P\|\le \lambda C$. Moreover, if $(e_n)$ is shrinking, then $\|P\|\le C$.
\end{prop}

\vspace{.1 in}

\noindent Using \cite[Proposition 3.7]{StR}, we deduce the following proposition.

\begin{prop}\label{L3.7}
Assume that $F$ is a Haar system space and $(\Tilde{h}_j)_{j=1}^\infty$ is a faithful Haar system in $F$.
Then the space $G = \overline{\spanS(\Tilde{h}_j: j\in \mathbb{N})}$ is a complemented subspace of $F$ and $P: F\to G$, determined by

\begin{equation}
P(f) = \sum\limits_{j=1}^\infty\Big\langle f, {\dfrac{\Tilde{h}_j}{\|h_j\|_{F^\ast}}}\Big\rangle {\dfrac{\Tilde{h}_j}{\|h_j\|_F}}, \; f\in F,
\end{equation}

\noindent is a well defined bounded projection onto $G$ with $\|P\|=1$.
\end{prop}

%%%%%%%%%%%
%%%%%%%%%%%
%%%%%%%%%%%

\section{Approximate factorization of a diagonal operator with a large diagonal through a bounded operator with a large diagonal}

\noindent The proof of Theorem \ref{commut} of this section uses the following result known as ``Small Perturbation Lemma".

\begin{thm}\cite[Theorem 4.23]{FHH}
Let $(x_n)$ be a basic sequence in a Banach space $X$, and let $(x_n^\ast)$ be the coordinate functionals $($they are elements of $\overline{\spanS(x_j: j\in \mathbb{N})}^\ast)$ and assume that $(y_n)$ is a sequence in $X$ such that

\begin{equation}
c = \sum\limits_{n=1}^\infty \|x_n - y_n\|\cdot \|x_n^\ast\| < 1.
\end{equation}

\noindent Then $(y_n)$ is also basic in $X$ and isomorphically equivalent to $(x_n)$, and if $\overline{\spanS(x_j: j\in \mathbb{N})}$ is complemented in $X$, then so is $\overline{\spanS (y_j: j\in \mathbb{N})}$.
\end{thm}

\noindent We note that the Rademacher sequence $(r_n)_{n\in \mathbb{N}\cup\{0, \emptyset\}} = \Big(\sum\limits_{I\in\mathcal{D}_n} h_I\Big)_{n\in \mathbb{N}\cup\{0, \emptyset\}}$ is a block basis of the Haar basis $(h_I)_{I\in\mathcal{D}^+}$, where we let $r_\emptyset = 1$. From now on we assume that $F$ is a Haar system space with the following property:

\vspace{.1 in}

\noindent $(\star)$ the Rademacher functions $(r_n)_n$ are weakly null.

\vspace{.1 in}

\noindent This condition is equivalent with the condition that the sequence $(r_n)_n$ is not equivalent to the $\ell_1$ unit vector basis \cite[Remark 2.15]{Prim}.
For $n\in \mathbb{N}$ and $\theta_n = \Big(\theta_n(I)\Big)_{I\in \mathcal{D}_n} \subset \{\pm 1\}$, we define $r^{\theta_n}_n = \sum\limits_{I\in \mathcal{D}_n} \theta_n(I) h_I$.
 
\begin{lem}\label{mult}
Let $A$ be a finite union of elements in $\mathcal D_k$, with $n>k$, so that $A$ is in $\mathcal D^+$.
Then it follows that

\begin{equation}
1_A r^{\theta_n}_n\overset{w}{\longrightarrow} 0, \;\;\; \mbox{as} \;\;\; n\to \infty.
\end{equation}*
\end{lem}

\begin{proof}

\noindent Recall that if $(x_n)_{n=1}^\infty$ is a bounded sequence in the Banach space $X$, then $(x_n)_{n=1}^\infty$ is weakly convergent to zero if and only if there exist convex combinations $\Big(\sum\limits_{j = k_{n-1} + 1}^{k_n}\alpha_j x_j\Big)_{n=1}^\infty$, with $\sum\limits_{j = k_{n - 1} + 1}^{k_n} \alpha_j = 1$, $\alpha_j \ge 0$, so that $\Big\|\sum\limits_{j = k_{n - 1} + 1}^{k_n} \alpha_j x_j\Big\|$ converges to $0$, as $n\to \infty$. Thus, from the assumption that $(r_n)$ is weakly null, it follows that there exist convex combinations $\Big(\sum\limits_{j = k_{n-1}+1}^{k_n} \alpha_j r_j\Big)_{n=1}^\infty$
such that $\sum\limits_{j = k_{n-1}+1}^{k_n} \alpha_j = 1$, $\alpha_j \ge 0$, and $\Big\| \sum\limits_{j = k_{n-1}+1}^{k_n} \alpha_j r_j\Big\|$ converges to $0$, as $n\to \infty$.
For the given $n$ in $\mathbb{N}$, let us denote

\begin{equation}
f_n = \sum\limits_{j = k_{n-1}+1}^{k_n} \alpha_j r_j, \hspace{3 cm} g_n = \sum\limits_{j = k_{n-1}+1}^{k_n} \alpha_j r^{\theta_j}_j.
\end{equation}

\noindent We write $f_n = \sum\limits_{j = k_{n-1}+1}^{k_n} \alpha_j r_j$ as $f_n = \sum\limits_{i=1}^\infty\beta^n_i h_i$. Then it follows that $g_n = \sum\limits_{j = k_{n-1}+1}^{k_n} \alpha_j r^{\theta_j}_j$ is equal to $g_n = \sum\limits_{i=1}^\infty\beta^n_i \Tilde{h}_i$, where $\Tilde{h}_i$ is equal to $-h_i$ or $h_i$. Since $(\Tilde{h}_i)_{i=1}^\infty$ is a faithful Haar system, it follows from Corollary \ref{Cor} that $\|f_n\| = \|g_n\|$. Thus, $r^{\theta_n}_n$ is weakly null.

\noindent Next, we want to see that $1_A r^{\theta_n}_n$ is weakly null for any $A\in \mathcal D^+$, where $A$ is a finite union of elements in $\mathcal D_k$, with $n>k$.

\noindent Let $A^c$ be the complement of $A$ in $[0,1]$. Since $1_A r_n^{\theta_n}+1_{A^c} r_n^{\theta_n} =  r_n^{\theta_n}$, it follows that $1_A r_n^{\theta_n}+1_{A^c} r_n^{\theta_n}$ is weakly null. Secondly we note that $1_A r_n^{\theta_n}- 1_{A^c} r_n^{\theta_n}$  is of the form $r_n^{\theta'_n}$ for some choice of $\theta'_n$, and thus $1_A r_n^{\theta_n}- 1_{A^c} r_n^{\theta_n}$ is also weakly null. Finally, we note that
$$
1_A r_n^{\theta_n}= {\frac{1}{2}}\Big(1_A r_n^{\theta_n}+1_{A^c} r_n^{\theta_n} +1_A r_n^{\theta_n}-1_{A^c} r_n^{\theta_n}\Big),
$$
and hence, $1_A r_n^{\theta_n}$ is weakly null.
\end{proof}

\begin{lem}\label{main.}
Let $F$ be a Haar system space, for which the Rademacher sequence $(r_n)_{n=1}^\infty = \Big(\sum\limits_{I\in \mathcal{D}_n} h_I\Big)_{n=1}^\infty$ is weakly null.

\noindent Let $T: F\to F$ be a linear bounded operator, for which there exists $\delta > 0$, so that for all $j\in \mathbb{N}$,
$$
\langle T h_j, h_j\rangle \ge \delta \|h_j\|_F \|h_j\|_{F^\ast} = \delta |I_j|.
$$  
\noindent Let $\eta > 0$ be given.

\noindent Then there exists a faithful Haar system $(\tilde h_j)_{j=1}^\infty$ of the following special form. For $j\ge 2$, $\tilde h_j = \sum\limits_{I\in \Delta_{I_j}} \theta_I h_I $, where $\Delta_{I_j}$ is a pairwise disjoint subset of $\mathcal{D}$ whose union has the same measure as $I_j$ and $(\theta_I)_{I\in\Delta_{I_j}} \subset \{\pm 1\}$, so that:
\begin{enumerate}
\item[(i)] $\sum\limits_{i=1}^\infty {\underset{j\neq i}{\sum\limits_{j=1}^\infty}} \Big|\Big\langle T\Big({\dfrac{\Tilde{h}_i}{\|h_i\|_F}}\Big), {\dfrac{\Tilde{h}_j}{\|h_j\|_{F^\ast}}} \Big\rangle\Big| < \eta$,

\item[(ii)]
$
\langle T \Tilde{h}_j, \Tilde{h}_j\rangle \ge \delta \|\Tilde{h}_j\|_F \|\Tilde{h}_j\|_{F^\ast} = \delta |I_j|$, 
where the equality follows from Proposition \ref{Prim.},
for all $j\in \mathbb{N}$. In particular, $T$ restricted to the faithful Haar system $(\Tilde{h}_j)_{j=1}^\infty$ has also a large diagonal.
\end{enumerate}
\end{lem}

\noindent To prove Lemma \ref{main.} we need the following probabilistic lemma.

\begin{lem}\label{Probabilistic}
\noindent Given $k, m$ in $\mathbb{N}$ with $k < m$, and $T$ as in Lemma \ref{main.}, let $\varepsilon = \Big(\varepsilon (J)\Big)_{J\in \mathcal{D}_m}\subset \{\pm 1\}$ be a family of Bernoulli variables, i.e., $\varepsilon = \Big(\varepsilon (J)\Big)_{J\in \mathcal{D}_m}$ are independent random variables on some probability space $(\Omega, \Sigma, \mathbb{P})$, and for given $m\in \mathbb{N}$ and $J\in \mathcal{D}_m$, it follows that $\mathbb{P}\Big(\varepsilon (J) = 1\Big) = \mathbb{P}\Big(\varepsilon (J) = - 1\Big) = {\frac{1}{2}}$.
Then
$$
\mathbb{E}\Biggl({\dfrac{\Big\langle \Tilde{h}^{\varepsilon}_{k, m}, T(\Tilde{h}^{\varepsilon}_{k, m})\Big\rangle}{2^{-k}}}\Biggr) \ge \delta,\;\;\mbox{for}\;\;\Tilde{h}_{k, m}^{\varepsilon} = \sum\limits_{J\in\Delta_k}\varepsilon (J) h_J = 1_{\Delta_k^\ast} \cdot r^{\varepsilon}_m,
$$
where $\Delta_k$ is a subset of $\mathcal{D}_m$, satisfying $|\Delta_k^\ast| = 2^{-k}$ with $\Delta^\ast_k = \cup_{J\in \Delta_k} J$.
\end{lem}

\begin{proof} [Proof of Lemma \ref{Probabilistic}]

\noindent We compute:

\begin{align*}
\mathbb{E} \bigg(\Big\langle\Tilde{h}_{k, m}^\varepsilon, T(\Tilde{h}_{k, m}^\varepsilon)\Big\rangle\cdot{\frac{1}{2^{-k}}}\bigg)
&= \mathbb{E} \bigg(\bigg\langle \sum\limits_{J'\in\Delta_k}\varepsilon (J') h_{J'}, \sum\limits_{J\in\Delta_k}\varepsilon (J) T(h_J)\bigg\rangle \cdot{\frac{1}{2^{-k}}}\bigg)\\
&= \mathbb{E} \bigg(\sum\limits_{J'\in\Delta_k} \sum\limits_{J\in\Delta_k} \varepsilon (J') \varepsilon (J) \Big\langle h_{J'}, T(h_J)\Big\rangle \cdot{\frac{1}{2^{-k}}}\bigg)\\
&= \sum\limits_{J\in\Delta_k} \Big\langle h_J, T(h_J)\Big\rangle {\frac{1}{2^{-k}}}\\
&\qquad+ \underset{J\neq J'}{\sum\limits_{J, J' \in \Delta_k}} \mathbb{E} \Big(\varepsilon (J) \varepsilon (J')\Big)\cdot \Big\langle h_{J'}, T(h_J)\Big\rangle {\frac{1}{2^{-k}}}\\
&= \sum\limits_{J\in\Delta_k} \Big\langle h_J, T(h_J)\Big\rangle {\frac{1}{2^{-k}}}\\ &
\ge \sum\limits_{J\in\Delta_k} \delta |J| {\frac{1}{2^{-k}}} \\
&= \delta {\frac{1}{2^{-k}}} \sum\limits_{J\in\Delta_k} |J| = \delta {\frac{1}{2^{-k}}} 2^{-k} = \delta.  \end{align*}

\end{proof}

\noindent Using Lemma \ref{Probabilistic}, we deduce that for all $m > k$ and disjoint collections $\Delta_k\subset\mathcal{D}$, satisfying $|\Delta^\ast_k| = 2^{-k}$, we can choose $\gamma_{k, m} = \Big(\gamma_{k, m}(J)\Big)_{J\in\Delta_k}\subset \{\pm 1\}$ so that the following is true:

\begin{equation}\label{one}
\Big\langle T(\Tilde{h}_{k, m}^{\gamma_{k, m}}), \Tilde{h}_{k, m}^{\gamma_{k, m}}\Big\rangle \ge \delta \cdot2^{-k}.
\end{equation}

\vspace{.2 in}

\begin{proof} [Proof of Lemma \ref{main.}]

\noindent For $m\in \mathbb{N}$ and $\theta = \Big(\theta (J)\Big)_{J\in \mathcal{D}_m}\subset \{\pm 1\}$, let $r^\theta_m$ be as in Lemma \ref{mult}. Let also $(\eta_i)_i$ be a sequence of positive numbers satisfying $\sum\limits_{i = k+1}^\infty \eta_i < \eta_k$, and chosen in a way that $\sum\limits_{i = 1}^\infty \eta_i < \eta$.

\noindent We choose

\vspace{.1 in}

\begin{equation}\label{cc.0}
\tilde h_1 = h_\emptyset = 1_{[0, 1]}.
\end{equation}

\vspace{.2 in}

\noindent For $j\ge 2$, we will choose inductively $\tilde h_j$ in the form
\begin{equation}\label{form}
\tilde h_j = \underset{I\subset \Delta^\ast_j}{\sum\limits_{I \in \mathcal{D}_{m_j}}} \theta_j(I) h_I,
\end{equation}
where $\Delta_j^\ast$ is a disjoint union of dyadic intervals, and $m_j > m_{j-1}$, together with $\underset{\hspace{1cm}I\subset \Delta_j^\ast}{\Big(\theta_j(I)\Big)_{I\in \mathcal{D}_{m_j}}} \subset \{\pm 1\} $, so that $\suppS(\tilde h_2) = \Delta^\ast_2 = [0, 1]$, and

\begin{equation}\label{c.0}
\Delta_j^\ast = \suppS(\tilde h_j) = [\tilde h_k=1], \;\mathrm{if}\; j=2k-1, \; j\ge 3,
\end{equation}

\begin{equation}\label{c.01}
\Delta_j^\ast = \suppS(\tilde h_j) = [\tilde  h_k=-1],\; \mathrm{if}\; j=2k,\; j\ge 4,
\end{equation}

\begin{equation}\label{c.3}
\sum\limits_{i=1}^{j-1} \Big|\Big\langle {\frac{T(\tilde h_i)}{\|h_i\|_F}},  {\frac{\tilde h_j}{\|h_j\|_{F^\ast}}} \Big\rangle\Big| <{\frac{\eta_j}{2}}, \end{equation}

\begin{equation}\label{c.4}
\sum\limits_{i=1}^{j-1}\Big|\Big\langle {\frac{T^\ast(\tilde h_i)}{{\|h_i\|}_{F^\ast}}}, {\frac{\tilde h_j}{\|h_j\|_F}}\Big\rangle\Big| = \sum\limits_{i=1}^{j-1} \Big|\Big\langle {\frac{\tilde h_i}{\|h_i\|_{F^\ast}}}, {\frac{T(\tilde h_j)}{\|h_j\|_F}}\Big\rangle\Big| < {\frac{\eta_j}{2}},
\end{equation}

\begin{equation}\label{c.6}
\bigg\langle {\frac{T^\ast(\tilde h_j)}{\|h_j\|_{F^\ast}}},  {\frac{\tilde h_j}{\|h_j\|_F}} \bigg\rangle > \delta.
\end{equation}

\vspace{.3 in}

\noindent Note that \eqref{c.3} and \eqref{c.4} will imply $(i)$ of Lemma \ref{main.} and \eqref{c.6} will imply $(ii)$, while \eqref{cc.0}, \eqref{form}, \eqref{c.0} and \eqref{c.01} will ensure that $(\tilde h_j)_j$ is a faithful Haar system.

\vspace{.1 in}

\noindent Since we have chosen $\Tilde{h}_1 = h_\emptyset = h_1 = 1$, it follows that \eqref{c.6} is true for $j=1$ and \eqref{c.3}, \eqref{c.4} are vacuous in that case.

\vspace{.1 in}

\noindent We now define $\tilde h_2$. For that purpose we first choose $n_2 > 0$ so that there exist finite linear combinations of Haar functions, denoted by $Y_1$  $(Y_1 \in \spanS(h_I: I \in \mathcal{D}^{n_2})\;)$, and so that

\begin{equation}\label{c.11}
\Big\|T(\tilde h_1) -  Y_1\Big\|_{F^\ast} < {\frac{\eta_2}{2}}.
\end{equation}

\vspace{.2 in}

\noindent Therefore, it is true that for $m > n_2$ and for all $\theta = \Big(\theta (I) \Big)_{I \in \mathcal{D}_m} \subset \{\pm 1\}$:

\begin{equation}\label{c.55}
\Big|\Big\langle Y_1,   \frac{r^\theta_m}{\|h_2\|_{F^\ast}} \Big\rangle\Big| = 0.
\end{equation}

\vspace{.1 in}

\noindent Together, \eqref{c.11} and \eqref{c.55}, imply that:

\begin{align*}
\Big|\Big\langle T(\tilde h_1),  \frac {r^\theta_m}{\|h_2\|_{F^\ast}} \Big\rangle\Big|& 
=  \Big|\Big\langle T(\tilde h_1) - Y_1 + Y_1, \frac{r^\theta_m}{\|h_2\|_{F^\ast}}\Big\rangle\Big|\\&
\le\Big|\Big\langle T(\tilde h_1) - Y_1,  \frac{r^\theta_m}{\|h_2\|_{F^\ast}}\Big\rangle\Big| + \Big|\Big\langle Y_1,  \frac{r^\theta_m}{\|h_2\|_{F^\ast}}\Big\rangle\Big|\\
&= \Big|\Big\langle T(\tilde h_1) - Y_1,  \frac{r^\theta_m}{\|h_2\|_{F^\ast}}\Big\rangle\Big| < {\dfrac{\eta_2}{2}}\cdot \bigg\|\frac{r^\theta_m}{\|h_2\|_{F^\ast}}\bigg\|_{F^\ast} < \dfrac{\eta_2}{2},
\end{align*}

\noindent which is true for any choice of $m > n_2$ and for all $\theta = \Big(\theta (I) \Big)_{I \in \mathcal{D}_m} \subset \{\pm 1\}$, and which holds because $r^\theta_m$ and $h_2$ have the same distribution. Thus, the inequality \eqref{c.3} holds for $j=2$ when we replace $\tilde h_2$ by $r^\theta_m$ for any choice of $m > n_2$ and any $\theta = \Big(\theta(I)\Big)_{I\in \mathcal{D}_m}\subset \{\pm 1\}$. 

\vspace{.1 in}

\noindent We also want to ensure \eqref{c.4} and \eqref{c.6}. For that purpose we refer to \eqref{one} (the consequence of Lemma \ref{Probabilistic}), and for $m > n_2$ we choose $\gamma_{2, m} = \Big(\gamma_{2, m}(I)\Big)_{I \in \mathcal{D}_m} \subset \{\pm 1\}$, so that

\begin{equation}
\bigg\langle {\frac{T^\ast(r^{\gamma_{2, m}}_m)}{\|h_2\|_{F^\ast}}},  {\frac{r^{\gamma_{2, m}}_m}{\|h_2\|_F}} \bigg\rangle > \delta.
\end{equation}

\noindent In addition, by Lemma \ref{mult} it follows that $r^{\gamma_{2, m}}_m \overset{w}{\longrightarrow} 0$, and hence, there exists $m_2 > n_2$, so that:

\begin{equation}\label{c.44}
\Big|\Big\langle T^\ast(\tilde h_1),  \frac{r^{\gamma_{2, m_2}}_{m_2}}{\|h_2\|_F} \Big\rangle\Big| < {\frac{\eta_2}{2}}.
\end{equation}

\noindent Thus \eqref{c.4} and \eqref{c.6} hold for $j=2$ with $\theta_2 = \gamma_{2, m_2}$, and $\tilde h_2 = \sum\limits_{I \in \mathcal{D}_{m_2}} \theta_2(I) h_I = \sum\limits_{I \in \mathcal{D}_{m_2}} \gamma_{2, m_2}(I) h_I$.

\vspace{.03 in}

\noindent Next, we define $\tilde h_j$ for $j\ge 3$, by assuming that $\tilde h_i$ for $i < j$ has been chosen satisfying the conditions \eqref{c.0}, \eqref{c.01}, \eqref{c.3}, \eqref{c.4}, and \eqref{c.6}. 

\vspace{.05 in}

\noindent We choose $n_j\ge m_{j-1}$ so that for $i = 2, \ldots, j-1$, there exist finite linear combinations of Haar functions, $Y_i$, with $Y_i \in \spanS(h_J: J \in \mathcal{D}^{n_j})$, so that

\begin{equation}\label{c.1}
\bigg\|{\frac{T(\tilde h_i)}{\|h_i\|_F}} - Y_i\bigg\|_{F^\ast} < {\frac{\eta_j}{2(j - 1)}},\;\; i = 2, \ldots, j-1.
\end{equation}

\noindent Then, for $k\ge 2$ and $j=2k - 1$, we take $\Delta_j^\ast = \Delta^\ast_{I^+_k} = [\tilde h_k = 1]$, and if $j=2k$, we choose $\Delta_j^\ast = \Delta^\ast_{I^-_k} = [\tilde h_k = -1]$. Thus, by this we want to ensure \eqref{c.0} and \eqref{c.01}.

\noindent Further, we notice that for $m > n_j$ and for all $\theta = \Big(\theta (J) \Big){\underset{J \in \mathcal{D}_m}{_{J\subset {\Delta^\ast_j}}}} \subset \{\pm 1\}$ it is true that:

\begin{equation}\label{c.5}
\sum\limits_{i=1}^{j-1} \Big|\Big\langle Y_i,  {\frac{1_{\Delta^\ast_j} r^\theta_m}{\|h_j\|_{F^\ast}}} \Big\rangle\Big| = 0,
\end{equation}

\noindent which, combined with \eqref{c.1}, yields the following computation:

\begin{align*}
\sum\limits_{i=1}^{j-1} \Big|\Big\langle {\frac{T(\tilde h_i)}{\|h_i\|_F}},  {\frac{1_{\Delta^\ast_j} r^\theta_m}{\|h_j\|_{F^\ast}}} \Big\rangle\Big|& 
= \sum\limits_{i=1}^{j-1} \Big|\Big\langle {\frac{T(\tilde h_i)}{\|h_i\|_F}} - Y_i + Y_i, {\frac{1_{\Delta^\ast_j} r^\theta_m}{\|h_j\|_{F^\ast}}}\Big\rangle\Big|\\&
\le\sum\limits_{i=1}^{j-1} \Big|\Big\langle {\frac{T(\tilde h_i)}{\|h_i\|_F}} - Y_i,  {\frac{1_{\Delta^\ast_j} r^\theta_m}{\|h_j\|_{F^\ast}}}\Big\rangle\Big|+ \sum\limits_{i=1}^{j-1} \Big|\Big\langle Y_i,  {\frac{1_{\Delta^\ast_j} r^\theta_m}{\|h_j\|_{F^\ast}}}\Big\rangle\Big|\\
&= \sum\limits_{i=1}^{j-1} \Big|\Big\langle {\frac{T(\tilde h_i)}{\|h_i\|_F}} - Y_i,  {\frac{1_{\Delta^\ast_j} r^\theta_m}{\|h_j\|_{F^\ast}}}\Big\rangle\Big| < {\frac{\eta_j}{2}}.
\end{align*}

\noindent Thus, the inequality \eqref{c.3} holds for $j$ when we replace $\tilde h_j$ by $1_{\Delta_j^\ast} r^\theta_m$ for any choice of $m > n_j$ and any $\theta = \Big(\theta(I)\Big)_{I\in \mathcal{D}_m}\subset \{\pm 1\}$.

\noindent Now we want to ensure \eqref{c.4} and \eqref{c.6}. We use \eqref{one}, which is a consequence of Lemma \ref{Probabilistic}, to choose $\gamma_{j, m} = \Big(\gamma_{j, m}(J)\Big){\underset{J \in \mathcal{D}_m}{_{J\subset {\Delta^\ast_j}}}} \subset \{\pm 1\}$, for $m > n_j$, so that

\begin{equation}
\Big\langle {\frac{T^\ast(1_{\Delta_j^\ast} r^{\gamma_{j, m}}_m)}{\|h_j\|_{F^\ast}}},  {\frac{1_{\Delta_j^\ast} r^{\gamma_{j, m}}_m}{\|h_j\|_F}} \Big\rangle > \delta.
\end{equation}

\noindent Moreover, Lemma \ref{mult} implies that $1_{\Delta_j^\ast} r^{\gamma_{j, m}}_m \overset{w}{\longrightarrow} 0$. Therefore, there exists $m_j > n_j$, so that:

\begin{equation}\label{c.44}
\sum\limits_{i=1}^{j-1} \Big|\Big\langle {\frac{T^\ast(\tilde h_i)}{\|h_i\|_{F^\ast}}},  {\frac{1_{\Delta_j^\ast} r^{\gamma_{j, m_j}}_{m_j}}{\|h_j\|_F}} \Big\rangle\Big| < {\frac{\eta_j}{2}}.
\end{equation}

\noindent Therefore, \eqref{c.4} and \eqref{c.6} hold with $\theta_j = \gamma_{j, m_j}$, and $ \tilde h_j = \underset{I\subset \Delta^\ast_j}{\sum\limits_{I \in \mathcal{D}_{m_j}}} \theta_j(I) h_I = \underset{I\subset \Delta^\ast_j}{\sum\limits_{I \in \mathcal{D}_{m_j}}} \gamma_{j, m_j}(I) h_I$.

\end{proof}

\begin{thm}[Approximate factorization of a diagonal operator with a large diagonal through a bounded operator with a large diagonal]\label{commut}

Let $F$ be a Haar system space, for which the Rademacher sequence $(r_n)_{n=1}^\infty = \Big(\sum\limits_{I\in \mathcal{D}_n} h_I\Big)_{n=1}^\infty$ is weakly null.
Let $T: F\to F$ be a bounded linear operator, so that for any $j\in \mathbb{N}$,
$$
\langle T h_j, h_j\rangle \ge \delta \|h_j\|_F \|h_j\|_{F^\ast} = \delta |I_j|.
$$

\noindent Let $\eta > 0$ be given, and let $(\Tilde{h}_j)_j$ be a faithful Haar system satisfying conditions $(i)$ and $(ii)$ of Lemma \ref{main.}. Then the diagonal operator $D: F\to F$ given by

\begin{equation}\label{D}
D(h_j) = \Big\langle T\bigg({\dfrac{\Tilde{h}_j}{\|h_j\|}_F}\bigg), {\frac{\Tilde{h}_j}{\|h_j\|_{F^\ast}}}\Big\rangle h_j, \; j\in \mathbb{N}
\end{equation}

\noindent is bounded and there exist bounded linear operators $A, B: F\to F$ with $\|A\| \|B\| \le 1$ so that $\|D - BTA\|\le 2 \eta$ and $\|D\|\le \|T\| + 2\eta$.
That is, the following diagram commutes with error $2 \eta$.

  \adjustbox{scale=1.5,center}   
 {
\begin{tikzcd}
F\arrow{r} {T} & F\arrow{d}{B}\\
F\arrow{u}{A} \arrow{r} [below]{D} & F\\
\end{tikzcd}
}

\end{thm}

\vspace{.03 in}

\begin{proof}
Let $G = \overline{\spanS(\Tilde{h}_j: j\in \mathbb{N})}$. It follows from Corollary \ref{Cor} that the operator $A: F\to G$, defined by

\begin{equation}
A(h_j) = \Tilde{h}_j,\; j\in \mathbb{N},
\end{equation}

\noindent is an isometry. Thus the inverse operator $A^{-1}: G\to F$, with $A^{-1}(\Tilde{h}_j) = h_j$, is well defined and is also an isometry.

\noindent Also, we deduce from Proposition \ref{L3.7} that $P: F\to G$, defined by
 
\begin{equation}
P(f) = \sum\limits_{j=1}^\infty\Big\langle f, {\frac{\Tilde{h}_j}{\|h_j\|_{F^\ast}}}\Big\rangle {\frac{\Tilde{h}_j}{\|h_j\|_F}}, \; f\;\mathrm{in}\; F,
\end{equation}

\noindent is a well defined projection onto $G$ with its norm satisfying $\|P\| = 1$. We denote $B = A^{-1} P$. It follows that for $j\in \mathbb{N}$:

\begin{align}\label{APTA} 
BTA (h_j)& = A^{-1} P T A (h_j)\\
&= A^{-1} P T (\Tilde{h}_j) = A^{-1} P \Big(T (\Tilde{h}_j)\Big)\notag\\
&= A^{-1} \Big( \sum\limits_{i=1}^\infty\Big\langle T (\Tilde{h}_j), {\frac{\Tilde{h}_i}{\|h_i\|_{F^\ast}}}\Big\rangle {\frac{\Tilde{h}_i}{\|h_i\|_F}}\Big)\notag\\
&= \sum\limits_{i=1}^\infty\Big\langle T (\Tilde{h}_j), {\frac{\Tilde{h}_i}{\|h_i\|_{F^\ast}}}\Big\rangle A^{-1}\Big({\frac{\Tilde{h}_i}{\|h_i\|_F}}\Big)\notag\\
&=\sum\limits_{i=1}^\infty\Big\langle T (\Tilde{h}_j), {\frac{\Tilde{h}_i}{\|h_i\|_{F^\ast}}}\Big\rangle {\frac{h_i}{\|h_i\|_F}}.\notag
\end{align}

\noindent Thus, by \eqref{D} and \eqref{APTA}, it follows for every $j$ in $\mathbb{N}$:

\begin{align}\label{diff}
&\|B T A (h_j) - D(h_j)\|
=\|A^{-1} P T A (h_j) - D(h_j)\|\\
=&\bigg\|\sum\limits_{i=1}^\infty\Big\langle T (\Tilde{h}_j), {\frac{\Tilde{h}_i}{\|h_i\|_{F^\ast}}}\Big\rangle{\frac{h_i}{\|h_i\|_F}} - \Big\langle T (\Tilde{h}_j), {\frac{\Tilde{h}_j}{\|h_j\|_{F^\ast}}}\Big\rangle {\frac{h_j}{\|h_j\|_F}} \bigg\|\notag\\
=&\bigg\|\underset{i\neq j}{\sum\limits_{i=1}^\infty}\Big\langle T (\Tilde{h}_j), {\frac{\Tilde{h}_i}{\|h_i\|_{F^\ast}}}\Big\rangle {\frac{h_i}{\|h_i\|_F}}\bigg\|\notag\\
\le&\underset{i\neq j}{\sum\limits_{i=1}^\infty}\bigg|\bigg\langle T (\Tilde{h}_j), {\frac{\Tilde{h}_i}{\|h_i\|_{F^\ast}}}\bigg\rangle\bigg|\cdot\bigg\|{\frac{h_i}{\|h_i\|_F}}\bigg\|\notag\\
=& \underset{i\neq j}{\sum\limits_{i=1}^\infty}\Big|\Big\langle T (\Tilde{h}_j), {\frac{\Tilde{h}_i}{\|h_i\|_{F^\ast}}}\Big\rangle\Big|.\notag
\end{align}

\noindent Let us take any
\begin{equation}
f = \sum\limits_{j=1}^\infty a_j {\dfrac{h_j}{\|h_j\|_F}}\;\mathrm{in}\; F,
\end{equation}
\noindent such that $\|f\|\le 1$.

\vspace{.1 in}

\noindent Since the Haar basis is monotone, $|a_n|\le 2$ for all $n$.

\noindent Thus, it follows from \eqref{diff} that:

\begin{align} \label{2lamet}
&\|B T A (f) - D(f)\| = \|A^{-1} P T A (f) - D(f)\|\\
\notag\le&\sum\limits_{j=1}^\infty {\dfrac{|a_j|}{\|h_j\|_F}} \cdot \|A^{-1} P T A (h_j) - D(h_j)\|\\
\notag\le&\sum\limits_{j=1}^\infty {\dfrac{|a_j|}{\|h_j\|_F}} \cdot \underset{i\neq j}{\sum\limits_{i=1}^\infty} \Big|\Big\langle T (\Tilde{h}_j), {\frac{\Tilde{h}_i}{\|h_i\|_{F^\ast}}}\Big\rangle\Big|\\
\notag\le&2 \sum\limits_{j=1}^\infty \underset{i\neq j}{\sum\limits_{i=1}^\infty} \Big|\Big\langle T \Big({\dfrac{\Tilde{h}_j}{\|h_j\|_F}}\Big), {\dfrac{\Tilde{h}_i}{\|h_i\|_{F^\ast}}}\Big\rangle\Big| < 2\eta,\\\notag
\end{align}

\noindent where the last inequality follows by Lemma \ref{main.}. This means that:

\begin{equation}\label{norm}
\|BTA - D\| = \|A^{-1} PTA - D\| \le 2 \eta.
\end{equation}

\noindent We have the following estimates for the product of the norms of operators $A$ and $B$ and for the norm of $D$:

\begin{equation}
\|A\| \|B\| = \|A\| \cdot \|A^{-1} P\| \le \|A\| \cdot \|A^{-1}\|\cdot \|P\| \le 1.
\end{equation}

\begin{equation}
\|D\| \le \|BTA\| + 2 \eta \le  \|B\| \cdot \|T\|\cdot \|A\| + 2 \eta \le \|T\| + 2 \eta.
\end{equation}

\end{proof}

\section{Factorization of the identity through a bounded operator with a large diagonal}

\begin{cor}[Factorization of the identity through a bounded operator with a large diagonal]\label{uncond.} Let $F$ be a Haar system space, for which the Haar system is an unconditional Schauder basis and assume that the Rademacher sequence $(r_n)_{n=1}^\infty = \Big(\sum\limits_{I\in \mathcal{D}_n} h_I\Big)_{n=1}^\infty$ is weakly null.

\noindent Let $T: F\to F$ be a bounded linear operator, so that for any $j\in \mathbb{N}$,
$$
|\langle T h_j, h_j\rangle| \ge \delta \|h_j\|_F \|h_j\|_{F^\ast} = \delta |I_j|.
$$  

\noindent Then there exist operators $A, B: F\to F$ so that the identity operator on $F$, which we denote by $I_F$, factors through $T$. That is, the following diagram commutes.

\adjustbox{scale=1.5,center}   
 {
\begin{tikzcd}
F\arrow{r} {T} & F\arrow{d}{B}\\
F\arrow{u}{A} \arrow{r} [below]{I_F} & F\\
\end{tikzcd}
}
\end{cor}

\begin{proof}
\noindent Without loss of generality we can assume that $\langle Th_i, h_i \rangle\ge 0$, because otherwise we pass to $\Tilde{T} = T\circ \Bar{D}$, for which the diagonal operator $\Bar{D}:F\to F$ is defined by $\Bar{D}(h_i) = \Big(\signS \langle Th_i, h_i \rangle\Big) h_i$. Since $(h_i)$ is unconditional, $\Bar{D}$ is bounded, and thus $\Tilde{T}$ is bounded. In this case we compute: 

\begin{align}\label{Til}
\langle \Tilde{T} h_j, h_j\rangle &= \langle T\circ \Bar{D} (h_j), h_j\rangle = \Big\langle T\Big( \signS \langle Th_j, h_j \rangle h_j\Big), h_j\Big\rangle\\ 
\notag&=\signS \langle Th_j, h_j \rangle \langle Th_j, h_j\rangle = |\langle Th_j, h_j\rangle|\ge\delta\|h_j\|_F\|h_j\|_{F^\ast},
\end{align}

\noindent and thus the claim follows from Theorem \ref{commut}.

\end{proof}

\section{Open questions}

\noindent We might want to check under what conditions the identity is factored through a bounded operator with a large diagonal when we do not assume unconditionality for the Haar system in a Haar system space. Also, we want to find other Banach function spaces for which the Haar system is a basis and satisfies the factorization property.

% ------------------------------------------------------------------------

\subsection*{Acknowledgment}
The author thanks Professor Thomas Schlumprecht for the helpful discussions and important suggestions.

% ------------------------------------------------------------------------
\end{document}